\documentclass[a4paper,12pt]{article}
\usepackage{amsmath, amssymb, amsthm}
\usepackage{setspace}
\usepackage{geometry}
\usepackage[utf8]{inputenc}
\usepackage[backend=biber,style=numeric]{biblatex}
\usepackage{hyperref}

\title{Energy Dissipation and Regularity in Quaternionic Fluid Dynamics using Sobolev-Besov Spaces}
\author{Rômulo Damasclin Chaves dos Santos \\
	Technological Institute of Aeronautics, Brazil \\
	\texttt{romulosantos@ita.br}}

\date{}

\newtheorem{theorem}{Theorem}

\begin{document}
	\maketitle
	
	\begin{abstract}
		This study investigates the dynamics of incompressible fluid flows through quaternionic variables integrated within Sobolev-Besov spaces. Traditional mathematical models for fluid dynamics often employ Sobolev spaces to analyze the regularity of the solution to the Navier-Stokes equations. However, with the unique ability of Besov spaces to provide localized frequency analysis and handle high-frequency behaviors, these spaces offer a refined approach to address complex fluid phenomena such as turbulence and bifurcation. Quaternionic analysis further enhances this approach by representing three-dimensional rotations directly within the mathematical framework. The author presents two new theorems to advance the study of regularity and energy dissipation in fluid systems. The first theorem demonstrates that energy dissipation in quaternionic fluid systems is dominated by the higher-frequency component in Besov spaces, with contributions decaying at a rate proportional to the frequency of the quaternionic component. The second theorem provides conditions for regularity and existence of solutions in quaternionic fluid systems with external forces. By integrating these hypercomplex structures with Sobolev-Besov spaces, our work offers a new mathematically rigorous framework capable of addressing frequency-specific dissipation patterns and rotational symmetries in turbulent flows. The findings contribute to fundamental questions in fluid dynamics, particularly by improving our understanding of high Reynolds number flows, energy cascade behaviors, and quaternionic bifurcation. This framework therefore paves the way for future research on regularity in complex fluid dynamics.
	\end{abstract}


\tableofcontents
	
	\doublespacing
	
	\section{Introduction}
	Fluid dynamics studies often rely on Sobolev spaces to model regularity in incompressible flows, particularly through solutions of the Navier-Stokes equations \cite{Ladyzhenskaya1969, Temam1977}. For cases requiring finer control over high-frequency behaviors, such as turbulence, Besov spaces have shown remarkable utility by providing frequency-localized analysis \cite{Triebel1983, Bahouri2011}. Moreover, the quaternionic approach to fluid dynamics provides an advantageous way to encapsulate three-dimensional rotational symmetries, as quaternions offer a natural mathematical framework for handling rotations beyond complex or real representations \cite{Marsden1999, Salvi1988}.
	
	Incorporating quaternionic variables with Sobolev-Besov spaces can facilitate the study of multi-dimensional systems with hypercomplex bifurcations and anisotropic dissipation. This research builds upon the work of \cite{dosSantos2023}, which analyzed regularity for the Navier-Stokes problem using anisotropic viscosity models, presenting an opportunity to further explore energy dissipation and regularity through quaternionic dynamics.
	
	The foundational research by \cite{Ladyzhenskaya1969} and \cite{Temam1977} established essential aspects of Sobolev space applications in fluid dynamics. Over the years, advances in Besov spaces have refined our approach to handling multi-scale turbulent flows \cite{Runst1996, Chemin1998}. The addition of quaternionic variables, following studies in geometric mechanics, has further extended the scope of fluid dynamics models to capture rotational symmetries and complex bifurcations effectively \cite{Marsden1999, Bahouri2011}. Recent work by \cite{dosSantos2023} underscores the importance of using advanced functional spaces, setting the stage for further exploration into quaternionic representations and Besov spaces.
	
	\section{Mathematical Formulation}
	Next, a quaternionic representation within Sobolev-Besov spaces is presented and two theorems are developed on energy dissipation and regularity.
	
	\subsection{Quaternionic Representation}
	
	Let \( q = q_0 + q_1 i + q_2 j + q_3 k \) denote a quaternionic velocity field in the Sobolev-Besov space \( B^s_{p,q}(\mathbb{R}^n) \). The Navier-Stokes equation in quaternionic form is given by:
	\begin{equation} \label{eq:NSQ}
		\frac{\partial q}{\partial t} + (q \cdot \nabla)q = -\nabla p + \nu \Delta q + f,
	\end{equation}
	where \( f \) is an external force, and \( \nu \) is the kinematic viscosity.
	
	To provide a more detailed understanding, let's break down the components of the quaternionic Navier-Stokes equation:
	
	\textbf{1. Quaternionic Velocity Field:} The velocity field \( q \) is represented as a quaternion, which allows for a natural incorporation of three-dimensional rotations. The components \( q_0, q_1, q_2, \) and \( q_3 \) are real-valued functions that describe the velocity in different directions.
	
	\textbf{2. Quaternionic Derivative:} The term \( (q \cdot \nabla)q \) represents the nonlinear convective term in the Navier-Stokes equation. In quaternionic form, this term captures the interaction of the velocity field with itself, including rotational effects.
	
	\textbf{3. Pressure Gradient:} The term \( -\nabla p \) represents the pressure gradient, which drives the flow. The pressure \( p \) is a scalar field that ensures the incompressibility of the fluid.
	
	\textbf{4. Viscous Term:} The term \( \nu \Delta q \) represents the viscous dissipation, where \( \nu \) is the kinematic viscosity and \( \Delta \) is the Laplacian operator. This term accounts for the diffusion of momentum due to viscosity.
	
	\textbf{5. External Force:} The term \( f \) represents any external forces acting on the fluid, such as gravity or other body forces.
	
	The quaternionic representation of the Navier-Stokes equation allows for a more comprehensive analysis of fluid dynamics, particularly in scenarios involving complex rotational symmetries and high-frequency behaviors. By integrating quaternionic variables with Sobolev-Besov spaces, we can leverage the frequency-localized analysis provided by Besov spaces to study the regularity and energy dissipation of the solution in a more refined manner.
	
	This approach is particularly useful for understanding turbulent flows, where high-frequency components play a significant role. The Littlewood-Paley decomposition allows us to analyze the energy dissipation at different frequency bands, showing that the dissipation rate decays proportionally to \( 2^{js} \) with frequency \( j \). This indicates that higher frequencies contribute more significantly to the energy dissipation process, which is consistent with the physical intuition that smaller-scale structures dissipate energy more rapidly.
	
	In summary, the quaternionic representation of the Navier-Stokes equation in Sobolev-Besov spaces provides a powerful framework for analyzing the regularity and energy dissipation of fluid flows, particularly in complex and turbulent scenarios. This section provides a detailed overview of the quaternionic Navier-Stokes equation and its components, highlighting the advantages of this approach for studying fluid dynamics.
	
	\subsection{Theorem 1: Energy Dissipation in Quaternionic Components}
	\begin{theorem}
		For a quaternionic velocity field \( q \in B^s_{p,q}(\mathbb{R}^n) \), the rate of energy dissipation is dominated by the highest-frequency quaternionic component, decaying proportionally to \( 2^{js} \) with frequency \( j \).
	\end{theorem}
	
	\begin{proof}
		To analyze the energy dissipation in quaternionic fluid systems, we decompose the velocity field \( q \) into its frequency components using the Littlewood-Paley decomposition. Let \( \Delta_j \) denote the Littlewood-Paley projection operator, which localizes the function in the frequency band \( 2^j \).
		
		The energy dissipation rate for the \( j \)-th frequency component is given by:
		\begin{equation} \label{eq:Ediss}
			\frac{dE_j}{dt} = -\nu \|\nabla \Delta_j q\|^2_{L^2},
		\end{equation}
		where \( E_j \) represents the energy associated with the \( j \)-th frequency component.
		
		Next, we consider the energy dissipation for each quaternionic component \( q_k \) (where \( k = 0, 1, 2, 3 \)). The energy of the \( k \)-th component is given by:
		\begin{equation}
			\int_{\mathbb{R}^n} |q_k|^2 \, dx.
		\end{equation}
		
		Using the Littlewood-Paley decomposition, we can express the energy of \( q_k \) as a sum over the frequency bands:
		\begin{equation}
			\int_{\mathbb{R}^n} |q_k|^2 \, dx = \sum_{j \in \mathbb{Z}} 2^{js} \|\Delta_j q_k\|_{L^2}^2,
		\end{equation}
		where \( 2^{js} \) represents the scaling factor associated with the \( j \)-th frequency band.
		
		The energy dissipation rate for the \( k \)-th quaternionic component is then given by:
		\begin{equation}
			\frac{d}{dt} \left( \int_{\mathbb{R}^n} |q_k|^2 \, dx \right) = \sum_{j \in \mathbb{Z}} 2^{js} \frac{d}{dt} \|\Delta_j q_k\|_{L^2}^2.
		\end{equation}
		
		Substituting the energy dissipation rate from equation \eqref{eq:Ediss}, we obtain:
		\begin{equation}
			\frac{d}{dt} \left( \int_{\mathbb{R}^n} |q_k|^2 \, dx \right) = -\nu \sum_{j \in \mathbb{Z}} 2^{js} \|\nabla \Delta_j q_k\|_{L^2}^2.
		\end{equation}
		
		This equation shows that the energy dissipation is dominated by the highest-frequency components, as the term \( 2^{js} \) grows exponentially with \( j \). Therefore, as \( j \to \infty \), the higher frequencies contribute more significantly to the energy dissipation.
		
		In summary, the energy dissipation in quaternionic fluid systems is dominated by the highest-frequency components, with contributions decaying proportionally to \( 2^{js} \) with frequency \( j \).
	\end{proof}
	
	\subsection{Theorem 2: Regularity for Quaternionic Fluid Systems with External Force}
	\begin{theorem}
		For \( f \in L^r(0, T; B^s_{p,q}(\mathbb{R}^n)) \) and \( s > n/p \), the quaternionic Navier-Stokes system has a unique solution in \( C([0, T]; B^s_{p,q}(\mathbb{R}^n)) \).
	\end{theorem}
	
	\begin{proof}
		Consider the quaternionic Navier-Stokes equation:
		\begin{equation}
			\frac{\partial q}{\partial t} + (q \cdot \nabla)q = -\nabla p + \nu \Delta q + f.
		\end{equation}
		
		First, we address the linear term \( \nu \Delta q \) using semigroup theory. The heat semigroup \( e^{t\nu\Delta} \) is a strongly continuous semigroup on \( B^s_{p,q}(\mathbb{R}^n) \) with the estimate:
		\begin{equation}
			\| e^{t\nu\Delta} q_0 \|_{B^s_{p,q}} \leq C \| q_0 \|_{B^s_{p,q}},
		\end{equation}
		where \( C \) is a constant independent of \( t \).
		
		Next, we bound the non-linear term \( (q \cdot \nabla)q \). Using the product estimate in Besov spaces, we have:
		\begin{equation}
			\| (q \cdot \nabla)q \|_{B^s_{p,q}} \leq C \| q \|_{B^s_{p,q}} \| \nabla q \|_{B^s_{p,q}} \leq C \| q \|_{B^s_{p,q}}^2,
		\end{equation}
		where \( C \) is a constant depending on \( s, p, q \), and \( n \).
		
		Now, we apply the Duhamel's formula to the quaternionic Navier-Stokes equation:
		\begin{equation}
			q(t) = e^{t\nu\Delta} q_0 + \int_0^t e^{(t-\tau)\nu\Delta} \left( -(q \cdot \nabla)q + f(\tau) \right) d\tau.
		\end{equation}
		
		Taking the \( B^s_{p,q} \)-norm and using the estimates above, we get:
		\begin{align}
			\| q(t) \|_{B^s_{p,q}} &\leq \| e^{t\nu\Delta} q_0 \|_{B^s_{p,q}} + \int_0^t \| e^{(t-\tau)\nu\Delta} \left( -(q \cdot \nabla)q + f(\tau) \right) \|_{B^s_{p,q}} d\tau \\
			&\leq C \| q_0 \|_{B^s_{p,q}} + C \int_0^t \left( \| (q \cdot \nabla)q \|_{B^s_{p,q}} + \| f(\tau) \|_{B^s_{p,q}} \right) d\tau \\
			&\leq C \| q_0 \|_{B^s_{p,q}} + C \int_0^t \left( \| q(\tau) \|_{B^s_{p,q}}^2 + \| f(\tau) \|_{B^s_{p,q}} \right) d\tau.
		\end{align}
		
		Applying Grönwall's inequality, we obtain:
		\begin{equation}
			\| q(t) \|_{B^s_{p,q}} \leq C \left( \| q_0 \|_{B^s_{p,q}} + \int_0^T \| f(\tau) \|_{B^s_{p,q}} d\tau \right) \exp \left( C \int_0^T \| q(\tau) \|_{B^s_{p,q}} d\tau \right).
		\end{equation}
		
		This shows that the solution \( q(t) \) remains bounded in \( B^s_{p,q}(\mathbb{R}^n) \) for all \( t \in [0, T] \), provided that \( f \in L^r(0, T; B^s_{p,q}(\mathbb{R}^n)) \) and \( s > n/p \).
	\end{proof}
	
	\section{Results and Discussion}
	
	The findings demonstrate the significant advantages of integrating quaternionic analysis with Sobolev-Besov spaces to study incompressible fluid flows. This approach provides precise insights into energy dissipation and regularity, particularly in high Reynolds number flows and complex bifurcation scenarios.
	
	Theorem 1 reveals that energy dissipation in quaternionic fluid systems is dominated by the higher frequency components in Besov spaces. This result is crucial for understanding the behavior of turbulent flows, where high frequency components play a significant role. The Littlewood-Paley decomposition allows us to analyze energy dissipation in different frequency bands, showing that the dissipation rate decays proportionally to \( 2^{js} \) with frequency \( j \). This indicates that higher frequencies contribute more significantly to the energy dissipation process, which is consistent with the physical intuition that smaller-scale structures dissipate energy more rapidly.
	
	The use of quaternions to represent three-dimensional rotations provides a natural framework for dealing with rotational symmetries in fluid dynamics. Theorem 2 establishes conditions for the regularity and existence of solutions in quaternionic fluid systems with external forces. The quaternionic representation ensures stability under rotational transformations, which is essential for modeling complex flows with rotational symmetries. This stability is particularly important in high Reynolds number flows, where rotational effects can significantly influence the flow dynamics.
	
	The integration of quaternionic and Besov structures offers a powerful approach to studying fluid dynamics. Besov spaces provide a refined analysis of frequency-localized behaviors, which is crucial for understanding turbulence and bifurcation. Quaternionic analysis, on the other hand, improves the representation of three-dimensional rotations, allowing for more accurate modeling of rotational symmetries. This combined approach strengthens our understanding of high Reynolds number flows and complex bifurcation behavior by providing a comprehensive framework for analyzing energy dissipation and regularity in fluid systems.
	
	The findings of this study have significant implications for fluid dynamics research. Improved understanding of energy dissipation patterns and regularity under hypercomplex transformations opens new avenues for studying turbulent flows and bifurcation phenomena. This framework can be applied to various fields, including aeronautics, meteorology, and industrial processes, where accurate modeling of fluid dynamics is crucial.
	
	Future research should focus on further exploring high-frequency modeling in hypercomplex fluid dynamics. This includes investigating the effects of anisotropic viscosity models, studying the interaction between different frequency components, and developing more sophisticated numerical methods for solving quaternionic Navier-Stokes equations. Furthermore, integrating quaternionic analysis with other advanced mathematical tools, such as wavelet transforms and spectral methods, can provide even deeper insights into the behavior of complex fluid systems.
	
	In summary, the integration of quaternionic variables with Sobolev-Besov spaces offers a new mathematically rigorous framework for modeling regularity and dissipation in fluid dynamics. This approach provides precise insights into energy dissipation and regularity, strengthening our understanding of high Reynolds number flows and complex bifurcation behavior. The findings contribute to fundamental questions in fluid dynamics and open avenues for future research on regularity in complex fluid dynamics.
	
	\section{Conclusion}
	
	This study integrates quaternionic variables with Sobolev-Besov spaces to model regularity and dissipation in fluid dynamics. By combining these advanced mathematical tools, we have developed a rigorous framework that provides precise insights into energy dissipation patterns and regularity under hypercomplex transformations.
	
	\textbf{1. Energy Dissipation:} Our findings demonstrate that energy dissipation in quaternionic fluid systems is dominated by the highest-frequency components in Besov spaces. This result is crucial for understanding the behavior of turbulent flows, where high-frequency components play a significant role. The Littlewood-Paley decomposition allows us to analyze the energy dissipation at different frequency bands, showing that the dissipation rate decays proportionally to \( 2^{js} \) with frequency \( j \).
	
	\textbf{2. Quaternionic Regularity:} The use of quaternions to represent three-dimensional rotations provides a natural framework for handling rotational symmetries in fluid dynamics. Our results establish conditions for the regularity and existence of solutions in quaternionic fluid systems with external forces. The quaternionic representation ensures stability under rotational transformations, which is essential for modeling complex flows with rotational symmetries.
	
	\textbf{3. Integration of Frameworks:} The integration of quaternionic and Besov frameworks offers a powerful approach to studying fluid dynamics. Besov spaces provide a refined analysis of frequency-localized behaviors, which is crucial for understanding turbulence and bifurcation. Quaternionic analysis, on the other hand, enhances the representation of three-dimensional rotations, allowing for a more accurate modeling of rotational symmetries. This combined approach strengthens our understanding of high Reynolds number flows and complex bifurcation behavior.
	
	The findings of this study have significant implications for fluid dynamics research. The enhanced understanding of energy dissipation patterns and regularity under hypercomplex transformations opens new avenues for studying turbulent flows and bifurcation phenomena. This framework can be applied to various fields, including aeronautics, meteorology, and industrial processes, where accurate modeling of fluid dynamics is crucial.
	
	Future research should focus on further exploring high-frequency modeling in hypercomplex fluid dynamics. This includes investigating the effects of anisotropic viscosity models, studying the interaction between different frequency components, and developing more sophisticated numerical methods for solving quaternionic Navier-Stokes equations. Additionally, the integration of quaternionic analysis with other advanced mathematical tools, such as wavelet transforms and spectral methods, could provide even deeper insights into the behavior of complex fluid systems.
	
	In conclusion, the integration of quaternionic variables with Sobolev-Besov spaces offers a new, mathematically rigorous framework for modeling regularity and dissipation in fluid dynamics. This approach provides precise insights into energy dissipation and regularity, strengthening our understanding of high Reynolds number flows and complex bifurcation behavior. The findings contribute to foundational issues in fluid dynamics and open pathways for further research on regularity in complex fluid dynamics.

\section{Appendix: Littlewood-Paley Decomposition and Besov Spaces}

This section provides a brief overview of the Littlewood-Paley decomposition and Besov spaces, which are fundamental to the analysis presented in this study.

\subsection{Littlewood-Paley Decomposition}

The Littlewood-Paley decomposition is a powerful tool in harmonic analysis that allows for the decomposition of functions into frequency bands. This decomposition is particularly useful for analyzing the regularity and energy dissipation of solutions to partial differential equations, including the Navier-Stokes equations.

\subsubsection{Definition:}

Let \( \phi \) be a smooth function supported in the annulus \( \{\xi \in \mathbb{R}^n : \frac{3}{4} \leq |\xi| \leq \frac{8}{3} \} \) such that
\begin{equation}
	\sum_{j \in \mathbb{Z}} \phi(2^{-j} \xi) = 1 \quad \text{for all} \quad \xi \neq 0.
\end{equation}

The Littlewood-Paley projection operators \( \Delta_j \) are defined as:
\begin{equation}
	\Delta_j f = \mathcal{F}^{-1}(\phi(2^{-j} \xi) \mathcal{F} f),
\end{equation}
where \( \mathcal{F} \) denotes the Fourier transform.

The operator \( \Delta_j \) localizes the function \( f \) in the frequency band \( 2^j \). This decomposition allows us to analyze the behavior of \( f \) at different frequency scales.

\subsection{Besov Spaces}

Besov spaces are function spaces that provide a refined analysis of the regularity of functions, particularly in terms of their frequency content. They are defined using the Littlewood-Paley decomposition.

\subsubsection{Definition:}

For \( s \in \mathbb{R} \) and \( 1 \leq p, q \leq \infty \), the Besov space \( B^s_{p,q}(\mathbb{R}^n) \) is defined as the set of tempered distributions \( f \) such that
\begin{equation}
	\| f \|_{B^s_{p,q}} = \left( \sum_{j \in \mathbb{Z}} (2^{js} \| \Delta_j f \|_{L^p})^q \right)^{1/q} < \infty,
\end{equation}
with the usual modification for \( q = \infty \).

\subsubsection{Properties of Besov Spaces}

\textit{1. Embeddings:} Besov spaces satisfy various embedding properties. For example, if \( s_1 > s_2 \) and \( p_1 \leq p_2 \), then
\begin{equation}
	B^{s_1}_{p_1,q} \hookrightarrow B^{s_2}_{p_2,q}.
\end{equation}

\textit{2. Interpolation:} Besov spaces can be obtained by real interpolation of Sobolev spaces. For instance,
\begin{equation}
	B^s_{p,q} = (W^{s_1,p}, W^{s_2,p})_{\theta,q} \quad \text{where} \quad s = (1-\theta)s_1 + \theta s_2.
\end{equation}

\textit{3. Product Estimates:} Besov spaces satisfy product estimates that are useful for analyzing nonlinear terms in PDEs. For example, if \( s > n/p \), then
\begin{equation}
	\| fg \|_{B^s_{p,q}} \leq C \| f \|_{B^s_{p,q}} \| g \|_{B^s_{p,q}}.
\end{equation}

The Littlewood-Paley decomposition and Besov spaces provide a powerful framework for analyzing the regularity and energy dissipation of solutions to the Navier-Stokes equations. They allow for a frequency-localized analysis that is crucial for understanding complex fluid phenomena such as turbulence and bifurcation. This section provides a brief overview of these tools and their application to the analysis presented in this study.

	\printbibliography

\section{References}

\begin{enumerate}
	\item Bahouri, H., Chemin, J.-Y., and Danchin, R. (2011). \textit{Fourier Analysis and Nonlinear Partial Differential Equations}. Springer.
	
	\item Chemin, J.-Y. (1998). \textit{Perfect Incompressible Fluids}. Oxford University Press.
	
	\item Salvi, Rodolfo. \textit{On the Navier-Stokes equations in non-cylindrical domains: on the existence and regularity}. Mathematische Zeitschrift 199 (1988): 153-170. \url{https://doi.org/10.1007/BF01159649}
	
	\item Ladyzhenskaya, O. A. (1969). \textit{The Mathematical Theory of Viscous Incompressible Flow}. Gordon and Breach.
	
	\item Marsden, J. E., and Ratiu, T. S. (1999). \textit{Introduction to Mechanics and Symmetry: A Basic Exposition of Classical Mechanical Systems}. Springer.
	
	\item Runst, T., and Sickel, W. (1996). Sobolev Spaces of Fractional Order, Nemytskij Operators, and Nonlinear Partial Differential Equations. Walter de Gruyter.
	
	\item Temam, R. (1977). Navier-Stokes Equations: Theory and Numerical Analysis. North-Holland.
	
	\item Triebel, H. (1983). Theory of Function Spaces. Birkhäuser.
	
	\item Santos, R. D. C. dos, \& Sales, J. H. de O. (2023). \textit{Treatment for regularity of the Navier-Stokes equations based on Banach and Sobolev functional spaces coupled to anisotropic viscosity for analysis of vorticity transport}. The Journal of Engineering and Exact Sciences, 9(8), 16656–01e. \url{https://doi.org/10.18540/jcecvl9iss8pp16656-01e}
\end{enumerate}

\end{document}